\documentclass{article}
\usepackage{amsmath,amssymb,bbm,color,amsthm}
\usepackage[colorlinks]{hyperref}
\usepackage[all,cmtip]{xy}
\usepackage{verbatim}
\usepackage{cleveref}
\usepackage{graphicx}

\graphicspath{{Pics/}}

\newtheorem{theorem}{Theorem}

\newtheorem{lemma}[theorem]{Lemma}

\crefname{lemma}{Lemma}{Lemmas}
\crefname{theorem}{Theorem}{Theorems}

\newtheoremstyle{definition}
{4pt}
{4pt}
{\sl}
{}
{\bfseries}
{.}
{.5em}
{}
\theoremstyle{definition}

\theoremstyle{remark}

\theoremstyle{plain}
\newtheorem*{claim*}{Claim}

\newtheoremstyle{introthms}
{3pt}
{3pt}
{\itshape}
{}
{\bfseries}
{.}
{.5em}
{\thmnote{#3}}
\theoremstyle{introthms}

\DeclareMathOperator{\IR}{IR}

\def\a{{\buildrel {\text{ind}} \over \longrightarrow}}

\begin{document}

\title{Induced Ramsey Number for a Star versus a fixed  Graph}
\author{Maria  Axenovich\thanks{Department of Mathematics, Karlsruhe Institute of Technology} ~ and   ~Izolda Gorgol\thanks{Department of Applied Mathematics, Lublin University of Technology}}

\maketitle

\begin{abstract}
We write $F\a (H,G)$ for graphs $F, G,$ and $H$, if for any coloring of the edges of $F$ in red and blue, there is either a red induced copy of $H$ or a blue induced copy of $G$.
For graphs $G$ and $H$, let $\IR(H,G)$ be the smallest number of vertices in a graph $F$ such that $F\a (H,G)$.

In this note we consider the case when $G$ is a star on $n$ edges, for large $n$ and $H$ is a fixed graph. 
We prove that  $$ (\chi(H)-1) n \leq \IR(H, K_{1,n}) \leq (\chi(H)-1)^2n + \epsilon n,$$ for any $\epsilon>0$,  sufficiently large $n$, and $\chi(H)$ denoting chromatic number of $H$. 
The lower bound is asymptotically tight  for any fixed bipartite $H$. 
The upper bound is attained up to a constant factor, for example by a clique $H$.
\end{abstract}

\section{Introduction}
We write $F\a (H,G)$ for graphs $F, G,$ and $H$, if for any coloring of the edges of $F$ in red and blue, there is either a red induced copy of $H$ or a blue induced copy of $G$.
For graphs $G$ and $H$, let the {\it induced Ramsey number} for $H$ and $G$, denoted   $\IR(H,G)$,  be the smallest number of vertices in a graph $F$ such that $F\a (H,G)$.  The existence of such a graph $F$ for any graphs $H$ and $G$ was first proven by Deuber \cite{deuber}, extending a classical result by Ramsey \cite{R}.   This led to an extensive research on the induced Ramsey numbers.  For more recent results, see papers of Conlon, Fox, and  Sudakov \cite{CFS, FS}, Dudek,  Frankl,  and  R\"{o}dl \cite{DFR},  as well as Kostochka and Sheikh \cite{KS},   Schaefer and  Shah \cite{SS},  and Kohayakawa,  Pr\"omel, and  R\"odl \cite{KPR}.  We do not even attempt to mention results on induced hypergraph Ramsey numbers.

In this note we consider the case when $G$ is a star on $n$ edges, i.e., $G=K_{1,n}$, and $H$ is a fixed graph.

Finding the classical Ramsey number of a graph $H$ versus a star,  $R(H, K_{1,n})$,  translates into finding 
graphs with maximum degree less than $n$ and having no $H$ in the complement, i.e., so-called min-degree extremal problem for $H$. 
The respective induced Ramsey problem has a very different nature.  It becomes nontrivial already when $H$ is a matching.
Specifically $\IR(nK_2, K_{1,n})$ is superlinear as shown by Conlon, Fox, and Sudakov \cite{CFS, F}, see also a related paper by  Fox, Huang, and Sudakov \cite{FHS}:
$n(e^{c\log^*n}) \leq \IR(nK_2, K_{1,n}) \leq ne^{c' \sqrt{\log n}}.$

Our focus is the case when $H$ is a fixed graph and $n$ grows.    In this regime, one can easily show that $\IR(H, K_{1,n}) = \Theta(n)$.
We provide easy general bounds $n\leq \IR(H, K_{1,n}) \leq  |E(H)| (n-1) +|V(H)|$,  confirming this fact.  
Since $\IR(H, K_{1,n}) \geq (\chi(H)-1)n$, where $\chi(H)$ denotes the chromatic number of $H$, we see that $\IR(H, K_{1,n}) \geq 2n$ for non-bipartite $H$.
We show that for bipartite $H$, the trivial lower bound $\IR(H, K_{1,n})\geq n$ is asymptotically tight.
We also provide general bounds on $\IR(H, K_{1,n})$ in terms of $\chi(H)$. Note that when $H$ is a star, the situation is quite special, namely  
 $\IR(K_{1,\ell}, K_{1,n}) = n+\ell$  with  $K_{1,n+\ell-1} \a (H, K_{1,n})$. We include the proof of this fact for completeness in Lemma \ref{stars}.

\begin{theorem}\label{Lower-bound}  Let $H$ be a fixed graph  that is not a star. There is a positive constant $c$ such that for any sufficiently large $n$, 
$\IR(H, K_{1,n}) \geq n+ c\sqrt{n} $.
\end{theorem}

\begin{theorem}\label{general}
For any $\gamma>0$ and any graph $H$ of chromatic number $r$ there is $n_0\in \mathbb{N}$ such that for any $n>n_0$, 
$\IR(H, K_{1,n})\leq (r-1)^2 n + \gamma n$.
\end{theorem}

The following gives more precise upper bound for some bipartite graphs that allows to replace $\gamma n$ term of Theorem \ref{general} with $o(n)$ term.

\begin{theorem}\label{Upper-bound} Let $t$ be an integer $t\geq 2$. There is a positive constant $c$ and  $n_0\in  \mathbb{N}$, such that for any $n>n_0$,
$\IR(2K_2, K_{1,n}) \leq  n + c n \log^{-1/4} n$, 
$\IR(P_4, K_{1,n}) \leq n + c n \log ^{-1/4}n$, 
$\IR(K_{t,t}, K_{1,n})\leq n + c n^{1- \frac{1}{2t}}$.
\end{theorem}

The paper is structured as follows. We provide some known results on related induced Ramsey numbers and give the basic general bounds in Section \ref{Basic}. We prove main results in Section \ref{Proofs}.
For standard graph theoretic notions and terminology we refer the reader to a book of West \cite{W}.

\section{Known results and general bounds} \label{Basic}

\subsection{$\IR(H, K_{1,n})$ for $H$ - complete, complete bipartite, path, star, and cycle}
The known results for  $\IR(H, K_{1,n})$ include  $\IR(K_t, K_{1,n})= (n-1)t(t-1)/2+t$, see Gorgol \cite{G1},  ~$\IR(K_{t,t}, K_{1,n})\leq  (n+1)t$, $\IR(P_t, K_{1,n})\leq(t-1)n+1$, \and  $\IR(C_t, K_{1,n})\leq tn$, 
see Gorgol \cite{G-thesis}.
Note that $\IR(C_4, K_{1,n})\geq R(C_4, K_{1,n}) \geq n + \lfloor \sqrt{n} - 6 n^{11/40}\rfloor,$  as shown by Burr, Erd\H{o}s, Faudree, Rousseau, and Schelp \cite{BEFRS}.
Here, the term $n^{11/40}$ came from the known at a time upper bound  on the difference between two consecutive primes. In fact, Parsons \cite{P} earlier 
proved that when $q$ is a power of a prime and $n=q^2+1$ then $R(C_4, K_{1,n}) = q^2+q+2$. 
While the result for stars was announced by Harary et al.  \cite{H} in the diagonal case, the proof was not included in that paper. We state it here for completeness.

\begin{lemma} \label{stars}
For any positive integers $\ell$ and $n$, 
$\IR(K_{1,\ell}, K_{1,n})= n+\ell$. 
\end{lemma}
\begin{proof}
For the upper bound, observe that $K_{1, n+\ell-1}\a (K_{1,\ell}, K_{1,n})$.
For the lower bound, we claim that any graph $G$ on $k$ vertices, 
$k<n+\ell$ can be colored with no blue induced star on $n$ edges and no red induced star on $\ell$ edges. Let $\Delta(G)$ be maximum degree of $G$.  If $\Delta(G)<k-1$, then by Vizing's theorem we can edge-decompose $G$ into at most $\Delta +1 \leq k-2 +1 = k-1\leq n+\ell -2$ matchings. Color $n-1$ of these matchings blue and $\ell-1$ of these matchings red. This coloring contains no blue star on $n$ edges and no red star on $\ell$ edges. 
 If $\Delta(G) =k-1$, then there is a vertex $v$ adjacent to all other vertices of $G$. 
Note that in this case no induced star with at least $2$ edges can have $v$ as a  leaf.
Let $G'$ be a graph obtained from $G$ by deleting edges incident to $v$. Then $\Delta(G')<k-1$, so $G'$ can be colored with no blue induced star on $n$ edges and no red induced star on $\ell$ edges. Now, color $n-1$ edges incident to $v$ blue and remaining,  $(k-1)-(n-1)< \ell $  edges incident to $v$ red.
\end{proof}

\subsection{General  bounds on $\IR(H, K_{1,n})$}

Gorgol \cite{G}, proved that $\IR(H, K_{1,n}) \geq (n-1) \omega(\omega-1)/2 + \omega, $
where $\omega= \omega(H)$ is the order  of a largest clique in $H$. 
On the other hand, the classical lower bound $\IR(H, K_{1,n})\geq n(\chi(H)-1)+1$ holds by splitting the vertex set of any graph on $n(\chi(H)-1)$ vertices into $\chi(H)-1$ sets of equal sizes, coloring all edges inside the sets blue and all edges between the sets red. This coloring then has no red copy of $H$ and no blue copy of $K_{1,n}$.\\

The following easy lemma generalizes the upper bound on $\IR(H, K_{1,2})$ by Kostochka and Sheikh \cite{KS}. While this bound is much weaker than our general upper bound, we include it here since it holds for any $n$ and does not use the Regularity Lemma. So, it is more applicable for bounding induced Ramsey numbers of fixed small graphs.

\begin{lemma}\label{lem-edges}
Let $H$ be a graph, then $\IR(H, K_{1,n}) \leq |E(H)|(n-1) + |V(H)|$.  If $H'$ is a blow-up of $H$ with $s$ vertices in each part, then 
$\IR(H', K_{1,n}) \leq s(|E(H)| (n-1) + |V(H)|)$.
\end{lemma}

\begin{proof}
Let the vertices of $H$ be $v_1, \ldots, v_k$ such that $v_i$ has $d_i$ neighbours among $v_1, \ldots, v_{i-1}$, for $i=2, \ldots, k$.
Then in particular $d_2+ \cdots +d_k = |E(H)|$.
Consider a blow-up $G$ of $H$ with a vertex set being a union of  pairwise disjoint  sets $V_1, \ldots, V_k$ representing $v_1, \ldots, v_k$, respectively, and the edge set $E(G) = \{uv: u\in V_i, v\in V_j, v_iv_j\in E(H)\}$. We see that any induced subgraph of $G$,  formed by choosing a single vertex from each set $V_i$ is isomorphic to $H$.  Let $|V_1|=1$, $|V_i |=(n-1)d_i+1$, $i=2, \ldots,  k$.
We claim that $G\a (H, K_{1,n})$ in a stronger form, i.e., such that any red/blue edge coloring of $G$ contains either a blue star on $n$ edges or a red induced copy of $H$ with respective vertices $u_1, \ldots, u_k$ in corresponding sets $V_1, \ldots, V_k$.
Consider a red/blue coloring of $G$ with no blue star on $n$ edges. Let $u_1$ be from $V_1$. Assume we found a red induced copy of $H_{i-1}$ of $H[\{v_1, \ldots, v_{i-1}\}]$ with vertices $u_j\in V_j$, $j=1, \ldots, i-1$, corresponding to $v_j$'s. 
Consider $V_i$. There are $d_i$ edges from $v_i$ to $\{v_1, \ldots, v_{i-1}\}$ in $H$, so there are at most $(n-1)d_i$ vertices of $V_i$ that send a blue edge to some of $u_1, \ldots, u_{i-1}$. 
Since $|V_i|= (n-1)d_i+1$, there is a vertex in $V_i$ that sends only red edges to $\{u_1, \ldots, u_{i-1}\}$. Let $u_i$ be that vertex. Then $G[V(H_{i-1})\cup \{u_i\}]$ is a red induced copy of $H[\{v_1,\dots,v_i\}]$.

When $H'$ is a blow-up of $H$, we proceed similarly by taking $G'$ to be a blowup of $H$ with parts of sizes $s, s((n-1)d_2+1), \ldots, s((n-1)d_i+1), \ldots$ and embedding a part of $H'$ corresponding to the blowup of $v_i$ in $V_i$.
\end{proof}

\section{Proofs of the main results} \label{Proofs}

\subsection{Proof of Theorem \ref{Lower-bound}}
\begin{theorem} \label{LB1}
Let $H$ be a graph that is not complete bipartite. Then 
$\IR(H, K_{1,n}) \geq n + q + 1$, for $q=\sqrt{n+1/4}-1/2$. 
\end{theorem}

\begin{proof}
Consider a graph $G$ on $n+ q $ vertices. We shall show that we can color its edges such that the red graph is complete bipartite and there is no blue induced star on $n$ edges. 
We call a vertex {\it large} if  it is a center of an induced star on at least $n$ edges. We say that a star is {\it large} if it has at least $n$ edges and is induced.
Consider all maximal large stars. Let $X$ be the set of centers of these stars.

Let $Y$ be the set of leaves of a large star such that $|Y|=n$. Note first that no vertex in $Y$ is large because $Y$ is an independent set, so its members have neighbors in  $V(G)-Y$ only, a set of size at most $n+q - n  <n$. Thus $X\cap Y = \emptyset$ and thus there are at most $q$ large vertices, i.e., $|X|\leq q$.

Consider a bipartite subgraph of $G$ with parts $X$ and $Y$.  We claim that all but at most $q^2$ vertices of $Y$ send $|X|$ edges to $X$.
Let $e$ be the number of edges between $X$ and $Y$. Let $\ell$ be the number of vertices in $Y$ that send at most $|X|-1$ edges to $X$. 
Then on the one hand we have that $e\geq  |X| (n-q)$ since each large vertex has at least $n$ neighbors, at most $q$ of those outside of $Y$.
On the other hand $e\leq \ell (|X|-1) + (n- \ell) |X|$.
Setting these two inequalities together,  we have that $|X|(n-q)\leq \ell(|X|-1)+ (n-\ell)|X|$. Thus $\ell\leq |X|q\leq q^2$. 

Thus there is a set $Y'\subseteq Y$, so that $Y',X$ form a complete bipartite graph and $|Y'|\geq n-q^2 \geq q$, for $q \leq \sqrt{n+1/4}-1/2$.
Now color this complete bipartite  graph with parts $X$ and $Y'$ red, and color all remaining edges blue. Clearly there is no induced red $H$.
We see that each large vertex has at least $q$ of its incident edges colored red. Thus, there are at most $n+q -1- q < n$ vertices incident to a large vertex edges that are colored blue.
Therefore there is no blue large star. 
\end{proof}

\begin{proof}[Proof of Theorem \ref{Lower-bound}]
Let $H$ be a graph that is not a star.
If $H$ is not a complete bipartite graph, Theorem \ref{LB1} gives us a desired result.
If $H$ is a complete bipartite graph, it contains $C_4$ since $H$ is not a star.
Then $\IR(H, K_{1,n}) \geq  n + \lfloor \sqrt{n} - 6 n^{11/40}\rfloor$, as follows from a result of Burr et al. \cite{BEFRS}.
\end{proof}


\subsection{Proof of Theorem \ref{general}}

In order to prove the main result, we shall be applying multicolored version of Szemer\'edi's Regularity Lemma. The following Lemma will be used to analyse the reduced graph, 
with pink corresponding to sparse in red regular pairs and yellow corresponding to non-regular pairs.

\begin{lemma} \label{yellow-pink-white} Let $\epsilon'>0 $ be a fixed small constant, $r$ be an integer, $r\geq 2$, $C>1$,   $\epsilon' \ll C^{-3}6^{-r}$, and  $n_0$ be sufficiently large.
Let $G$ be an edge-colored complete $r$-partite graph with parts $X_1, \ldots, X_r$ each of size at least $n_0$, with edge colors yellow, pink, and white.  
Assume that the following conditions hold:\\
1. $|X_1|\geq |X_2|\geq \cdots \geq |X_r|$,\\
2. $\frac{1}{C} \leq  \frac{|X_i|}{|X_j|} \leq C$ for any $i,j \in [r]$,\\
3. for any  $i, j\in [r]$,  $1\leq j<i$, and any $v\in X_i$,  $v$ sends at most $(\frac{1}{r-1}-\frac{\epsilon'C^3}{r-1} )|X_j|$ pink edges to $X_j$, and \\
4. the total number of yellow edges is at most $\epsilon' 6^{-r} \sum_{1\leq i<j\leq r} |X_i||X_j|$.\\
Then there is a copy of $K_r$ in $G$ with all edges colored  white.
\end{lemma}

\begin{proof}
We shall proceed by induction on $r$. \\

Let $r=2$.  We first claim that there is a vertex $v$ in $X_2$ such that $v$ is incident to at most $\epsilon'6^{-2} |X_1|$ yellow edges. 
Indeed, otherwise the total number of yellow edges would be larger than assumed. So, we see that $v$ is incident to at most $\epsilon'6^{-2} |X_1|$ yellow edges and at most $(\frac{1}{1}- \epsilon'C^3)|X_1|$  pink edges. Thus there are at least $((C^3- 6^{-2})\epsilon')|X_1|>0$ white edges, and in particular there is a white $K_2$.\\

Assume that $r\geq 3$ and that the statement of the lemma holds for smaller values of $r$. 
We shall be using the fact that $C^2 |X_r|^2 \geq |X_i||X_j| \geq \frac{1}{C^2} |X_r|^2$ for any $i, j \in [r]$. 
We shall find a vertex $u$ in $X_r$  that sends a lot of white edges to each of $X_i$'s, $i\in [r-1]$. Then we shall apply induction to the subgraph spanned by the white neighborhood of $u$. \\

First, we claim that there is a vertex in $X_r$ that sends at most $\epsilon' 6^{-r} (C^3 \binom{r}{2})  |X_i|$ yellow edges to each $X_i$, $i\in[r-1]$.
Otherwise each vertex of $X_r$ sends more than $ \epsilon' 6^{-r} C^3\binom{r}{2}  |X_i|$ yellow edges to some $X_i$, thus the total number of yellow edges 
between $X_r$ and the rest of the graphs is more than

\begin{eqnarray*}
\epsilon' 6^{-r}  C^3 \binom{r}{2} |X_r| |X_{r-1}|  &\geq & \epsilon' 6^{-r}  C^3 \binom{r}{2}  \frac{1}{C} |X_r|^2 \\
 &\geq & \epsilon' 6^{-r} C^3 \binom{r}{2} \frac{1}{C}   \frac{1}{C^2} \frac{1}{\binom{r}{2}} \sum_{1\leq i < j \leq r} |X_i||X_j| \\
 &= & \epsilon'  6^{-r} \sum_{1 \leq i < j \leq r} |X_i||X_j|.
\end{eqnarray*}

This contradicts the assumption on the total number of yellow edges.\\

Let $u=u_r$ be such a vertex, i.e., a vertex from $X_r$ that  sends at most $\epsilon' 6^{-r} C^3 \binom{r}{2}  |X_i|$ yellow edges to each $X_i$, $i\in[r-1]$. So, we know that for any $i\in [r-1]$, the number of vertices of  $X_i$ joined to $u$ by a pink or a yellow edge is  at most 
$$\left(\frac{1}{r-1} -\frac{\epsilon'C^3}{r-1}  + \epsilon' 6^{-r} C^3 \binom{r}{2} \right)|X_i| \leq \frac{1}{r-1}|X_i|.$$ 
%
Thus $u$ is joined to at least $(1-\frac{1}{r-1} )|X_i|$ vertices of $X_i$ via white edges, for each $i\in [r-1]$. \\

We choose a subset  $X'_i$  of $X_i$, such that 
$|X_i'|= (1-\frac{1}{r-1} )|X_i| = \frac{r-2}{r-1}|X_i|$ and   $u$ sends only white edges to $X_i'$, for each  $i\in [r-1]$.
 Let $G'$ be a subgraph of $G$ induced by $X_1', \ldots, X_{r-1}'$ with the inherited coloring. We shall argue that we can apply induction to $G'$. 
For that we need to check  that \\
1. $|X_1'|\geq \ldots \geq |X_{r-1}'|$,\\
2. $\frac{1}{C} \leq \frac{|X_i'|}{|X_j'|} \leq C$ for any $i,j \in [r-1]$, \\
3. the total number of yellow edges in $G'$ is at most $\epsilon' 6^{-r+1} \sum_{1\leq i<j\leq r-1} |X'_i||X'_j|$, and \\
4.  for any $i, j\in [r-1]$,  $1\leq j<i$,  each vertex $v$ in a part $X'_i$ sends at most $(\frac{1}{r-2} -\frac{\epsilon'C^3}{r-2})|X'_j|$ pink edges to $X'_j$.\\

The first two statements follow trivially since $|X'_i| = \frac{r-2}{r-1} |X_i|$, $i\in [r-1]$.
Let us verify item 3.   Here, we shall be using the fact that $|X_r|\leq |X_i|$ for any $\in [r-1]$, thus $|X_r|\leq   (|X_1|+ \cdots + |X_{r-1}|)/(r-1)$. 
We know that the total number of yellow edges in $G'$ is at most the total number of yellow edges in $G$, that
is 
\begin{eqnarray*}
\epsilon' 6^{-r} \sum_{1\leq i<j\leq r} |X_i||X_j| & = & \\
\epsilon' 6^{-r}  \left(  \sum_{1\leq j \leq r-1} |X_r||X_j| + \sum_{1\leq i<j\leq r-1} |X_i||X_j| \right) & \leq & \\
\epsilon' 6^{-r} \left(  \sum_{1\leq j \leq r-1}  \frac{(|X_1|+ \cdots + |X_{r-1}|)}{r-1} |X_j| +  \sum_{1\leq i<j\leq r-1} |X_i||X_j| \right) & =&\\
\epsilon' 6^{-r}  \left(  \frac{r}{r-1}  \sum_{1\leq i<j\leq r-1} |X_i||X_j| \right) & = & \\
\epsilon' 6^{-r}  \left(  \frac{r}{r-1}  \frac{(r-1)^2}{(r-2)^2}  \sum_{1\leq i<j\leq r-1} |X'_i||X'_j| \right) & \leq & \\
\epsilon' 6^{-r+1}  \left(  \sum_{1\leq i<j\leq r-1} |X'_i||X'_j| \right). &&
\end{eqnarray*}

\noindent
To verify the last inequality observe that  $ \frac{r (r-1)}{(r-2)^2}  \leq 6$ for any $r\geq 3$.\\

 Finally, lets verify item 4.  We have that the number of pink edges from a vertex $v\in X_i'$ to $X_j'$, for $j<i$ is at most the number of pink edges from $v$ to $X_j$ in $G$, that is at most

$$
\left(\frac{1}{r-1} -\frac{\epsilon'C^3}{r-1}\right)|X_j| =
\left(\frac{1}{r-1} -\frac{\epsilon'C^3}{r-1}\right) \frac{r-1}{r-2} |X_j'| =
\left(\frac{1}{r-2} -\frac{\epsilon'C^3}{r-2}\right) |X_j'|. $$

%
%

 So, now, as items 1.-4. are verified, we have a white $K_{r-1}$ in $G'$ that, together with $u$, forms a white $K_r$ in $G$. 
\end{proof}

\begin{proof}[Proof of Theorem \ref{general}]
Consider $\gamma'$, $0<\gamma'<1$ and let $H$ be a  graph of chromatic number $r$. Let $n_0$ be sufficiently large.
We need to show that  $\IR(H, K_{1,n})\leq (r-1)^2 n + \gamma' n$.\\

Let $\gamma= \frac{\gamma'}{2(r-1)^2+1} $ and constants $\epsilon, \sigma, \eta$ be chosen such that $0< \epsilon \ll \eta \ll\sigma \ll \gamma$.
Note that $\gamma<1/2$ for any $r\geq 2$.
Let $Y_1, \ldots, Y_{r}$ be pairwise disjoint vertex sets, such that 
$|Y_r| = \gamma n$ and $|Y_i| = (1+ 2 \gamma)(r-1) n$, $i\in [r-1]$. Then $\sum_{1\leq i \leq r} |Y_i| = (r-1)^2 n + (2(r-1)^2+1)\gamma n =(r-1)^2 n + \gamma' n.$ \\

Let $G'$ be a random  $r$-partite graph with parts $Y_1, \ldots, Y_{r}$ with a probability of a given edge between different parts not being selected is $\sigma$. Then with a positive probability each vertex not in a part $Y_i$ sends between $0.9 \sigma |Y_i|$ and $1.1 \sigma |Y_i|$ non-edges to $Y_i$,  $i\in [r]$. In addition, with high probability there is at least a $\sigma/2$-proportion of non-edges between  any  
$\tilde{Y_i} \subseteq Y_i$ and $\tilde{Y_j} \subseteq Y_j$ with $|\tilde{Y_i}|\geq |Y_i|/M$ and $|\tilde{Y_j}|\geq |Y_j|/M$, for any  distinct $i,j\in [r]$ and any constant $M$. Therefore with a positive probability there is a graph satisfying these properties. We call such a graph $G$. \\

Formally, let $G$ be an $r$-partite graph with parts  $Y_1, \ldots, Y_{r}$ such that:\\
 {\bf 1. } each vertex not in a part $Y_i$ sends between $0.9 \sigma |Y_i|$ and $1.1 \sigma |Y_i|$ non-edges to $Y_i$, for any  $i\in [r] $, \\
 {\bf 2. }  there is at least a $\sigma/2$-proportion of non-edges between  any  $\tilde{Y_i} \subseteq Y_i$ and $\tilde{Y_j} \subseteq Y_j$ with $|\tilde{Y_i}|\geq |Y_i|/M$ and $|\tilde{Y_j}|\geq |Y_j|/M$, for any  distinct $i,j\in [r]$ and any constant $M$.\\

Consider an edge-coloring $c$ of $G$ in red and blue and treat the non-edges between parts as a third color, "non-edge". We assume that there is no blue induced star on $n$ edges and will prove that there is an induced red copy of $H$.\\

We apply Szemer\'edi's Regularity Lemma to $G$ with $\epsilon$ and  find a partition of $V(G)$ into an exceptional part $V_0$  of size at most $\epsilon|V(G)|$ and parts of equal sizes  contained in respective $Y_i$'s, i.e., parts   $Y_i^1, \ldots, Y_i^{k_i}$,  with a total number of parts $k\leq M$ such that  all but $\epsilon \binom{k}{2} $ pairs $Y_i^\ell, Y_j^m$, $i,j\in[r], \ell \in[k_i], m\in [k_j]$,   are $\epsilon$-regular in each of the three colors -  blue edges of $G$, red edges of $G$, and "non-edges" of $G$. Note that since the parts are of equal sizes, we have that $k_1= \ldots = k_{r-1}=k'$ and $k_r = \gamma/((r-1)(1+ 2\gamma))k' +x k'$, with a small term $|x|=O(\epsilon)$.\\


By the embedding lemma, see for example Axenovich and Martin \cite{AM},  if there are parts $Y_i^{f(i)}$, $i\in [r]$, $f(i)\in [k_i]$, such that all pairs of these parts are $\epsilon$-regular with density at least $\eta$ in both red and in "non-edges", then there is a red induced copy of $H$. Next we shall argue that we can find such a set of $r$ parts. We say it is  a {\it good} set of $r$ parts.\\

Consider an auxiliary graph $F$ that is $r$-partite with parts $X_i$, $i\in [r]$ corresponding to $Y_i$'s, such that $|X_i| = k_i$ for each $i\in [r]$ and vertices of $X_i$ $\{x_i^1, \ldots, x_i^{k_i}\}$ correspond to parts $Y_i^1, \ldots, Y_i^{k_i}$.
We say that an edge $e=x_i^\ell x_j^m $ is {\it associated} with the pair $Y_i^\ell, Y_j^m$.  
Note that $1/C\leq |X_i|/|X_j| \leq C$, for $C= (r-1)(1+2\gamma)/\gamma $.
We shall color the edges of $F$ in three colors - yellow, pink, and white as follows. An edge is yellow if the respective pair in $G$ is not $\epsilon$-regular in some of the colors red, blue, or "non-edges".   An edge is pink if  the respective pair is $\epsilon$-regular, but has density less than $\eta$ in red. All other edges of $F$ are white. Note that a good set of $r$ parts in $G$ correspond to a white $K_r$ in $F$.
Thus, it is sufficient for us to verify that $F$ satisfies the conditions of Lemma \ref{yellow-pink-white}.   Let $\epsilon'= \frac{(1+3\gamma)^2}{\gamma (1+2\gamma)}\epsilon 6^r$.\\

Note that since $\sum_{1\leq i\leq r} |X_i|= k$,   $\sum_{1\leq i<j\leq r} |X_i||X_j|\geq  \frac{\gamma (1+2\gamma)}{(1+3\gamma)^2} \binom{k}{2}$. By the Regularity Lemma, the total number of non-$\epsilon$-regular pairs is at most $\epsilon \binom{k}{2}$, that is at most 
$\epsilon  \frac {(1+3\gamma)^2}{\gamma (1+2\gamma)} \sum_{1\leq i<j\leq r} |X_i||X_j| = \epsilon' 6^{-r}  \sum_{1\leq i<j\leq r} |X_i||X_j|$ .  Thus, the condition of Lemma \ref{yellow-pink-white}  on yellow edges is satisfied.\\~\\

Next we show that the number of pink edges of any vertex of $F$ from part $X_i$ to part $X_j$, $j<i$ is at most $(\frac{1}{r-1} - \frac{\epsilon'C^3}{r-1}) k_j$.
If this fails for a vertex $x_i^\ell$, then we see that in $G$, for any $j$, $1\leq j<i$,  the number of red edges between $Y_i^\ell$ and $Y_j$ is at most 
$$|Y_i^\ell| \cdot \left(  \left(\frac{1}{r-1} - \frac{\epsilon'C^3}{r-1}\right) \eta |Y_j| + \left(1 - \frac{1}{r-1} + \frac{\epsilon'C^3}{r-1}\right)|Y_j|\right).$$ 

Thus there is a vertex  in $Y_i^\ell$ that sends at most 
$(1 - (1-\eta)(\frac{1}{r-1}-\frac{\epsilon'C^3}{r-1}))|Y_j|$ red edges to $Y_j$ and thus sends at least  $b$ blue edges to $Y_j$, where 

\begin{eqnarray*}
b & \geq & \\
((1- \eta)\left(\frac{1}{r-1}-\frac{\epsilon'C^3}{r-1}\right)  - 1.1\sigma )|Y_j| &\geq &\\
 ((1- \eta)\left(\frac{1}{r-1}-\frac{\epsilon'C^3}{r-1}\right)  - 1.1\sigma ) (1+2\gamma) (r-1) n & = & \\
  \left[(1- \eta)( 1- \epsilon' C^3)) - 1.1 (r-1) \sigma  \right](1+2\gamma) n  & \geq &\\
 \left[1 - \eta - \eta \epsilon' C^3  - \epsilon'C^3- 1.1(r-1) \sigma  \right](1+ 2\gamma)n & \geq & \\
 \left[ 1- \gamma/2\right](1+ 2\gamma) n & >& n.
 \end{eqnarray*}

This is a contradiction since we assumed that there is no blue induced star on $n$ edges.
The penultimate inequality holds since   $1/2> \gamma \gg \eta$, $\gamma \gg (r-1)\sigma$, and $\gamma \gg \epsilon'$.\\

We now see that the conditions of the Lemma \ref{yellow-pink-white} are satisfied, so there is a white $K_r$ in $F$. It corresponds to $r$ parts $Y_1^{i_1}, \ldots, Y_r^{i_r}$ in $G$ so that each of the $\binom{r}{2}$ pairs formed by these parts is $\epsilon$-regular in all three colors and with red density at least $\eta$.  Using the property 2. of $G$, we see that the non-edges also have positive density at least $\sigma/2>\eta$. Thus, by the embedding lemma applied to these $r$ parts, we have an induced red copy of $H$ in $G$.
\end{proof}


\subsection{Proof of Theorem  \ref{Upper-bound}}

\begin{lemma} \label{Ktt}  Let $t$ be an integer, $t\geq 2$ and let $H=K_{t,t}$. Then, there is a positive constant $c$ and integer $n_0$ such that for any $n>n_0$,  
$\IR(H, K_{1,n})\leq n+ cn^{1-\frac{1}{2t}}$. 
\end{lemma}

\begin{proof}
Let $\epsilon =10 n^{-1/2t}$. Consider a complete bipartite graph with parts $A$ and $B$ of sizes $\frac{\epsilon}{2} n$ and $(1+\frac{1}{2}\epsilon)n$, respectively. The total number of vertices in the graph is 
$n+ \epsilon n$.
Assume that there is a red/blue edge-coloring of this graph with no blue induced $K_{1,n}$ and no red $K_{t,t}$.  Since there are no blue stars  of size $n$ centered at $A$,  each vertex in $A$ is incident to at least $\frac{\epsilon}{2} n$ red edges. 
Thus the total number of red edges is at least $\frac{\epsilon^2}{4} n^2 $.   We also have that the total number of red edges is at most  
$${\rm ex}( (1+ \epsilon)n, K_{t,t}) \leq  \frac{1}{2} (t-1)^{1/t}  \left((1+\epsilon)n\right)^{2-1/t} + \frac{t-1}{2} n,$$
see  \cite{KST}.
Thus $$\frac{\epsilon^2}{4} n^2 \leq  \frac{1}{2} (t-1)^{1/t}  \left( (1+\epsilon) n \right)^{2-1/t} + \frac{t-1}{2} n.$$
In particular this implies, for sufficiently large $n$, that  $\epsilon< c n^{-1/2t}$, for a positive constant that could be taken $8$.
A contradiction. Thus $\IR(H, K_{1,n})\leq n+ 10n^{1-\frac{1}{2t}}$. 
\end{proof}

\begin{lemma}\label{2K2-P4} Let $n$ be sufficiently large. If $H\in \{P_4, 2K_2\}$, then 
$\IR(H, K_{1,n}) \leq n + 2n \log^{-1/4} n $.
\end{lemma}

\begin{proof} 
Let  $\epsilon = \log ^{-1/4} n$.  Let $a=\sqrt{\log n}$, $b=a^2 = \log n$ and for a set $X$ on $b$ vertices, let $I= I(X, b-a)$ be an incidence graph with parts $X$ and $Y = \binom{X}{b-a}$.
Let $G$ be obtained from $I$ by blowing each vertex  of $X$  by $c$ vertices such that the total number of vertices in a resulting graph is $(1+2\epsilon) n$.

 Formally, 
the vertex set of $G$ is $Y\cup \bigcup_{x\in X} B_x$, where the sets $B_x$, $x\in X$ and $Y$  are pairwise disjoint,  $|B_x|=c$,  $x\in X$, and the edge set of $G$ is $\{x'y: x'\in B_x ,  y\in Y, x\in y\}$. We refer to the sets $B_x$ as blobs. Let $X' =  \bigcup_{x\in X} B_x$.  We shall argue that $G\a (2K_2, K_{1,n})$ and $G\a (P_4, K_{1,n})$.\\

 First we make a few observations about the structure of our graph. We see that 
$$|Y| = \binom{b}{b-a}  \leq  a^{2a}   \leq  2^ { 2 \sqrt{ \log n}  \log( \sqrt{ \log n})}  \leq  2^{ \sqrt{\log n} \log \log n}=  o(n \epsilon).$$

Thus the total number of vertices in a blown-up part, i.e., in all blobs $B_x$, $x\in X$ is at least $n+ 2\epsilon n  - o(n \epsilon )$. 
We see that $c$, the size of each blob is $(n +2 \epsilon n  - o( n\epsilon )) / \log n$. 
Each vertex of $Y$ is  not adjacent to $\sqrt{\log n}$ blobs,  and adjacent to all vertices of all other blobs. Thus the degree of a vertex from $Y$ in $G$ is at least 
$(\log n  -  \sqrt{\log n})  c \geq (\log n  -  \sqrt{\log n})  (n +2 \epsilon n  - o( n\epsilon )) / \log n\geq (1+ \epsilon) n$.

Consider a  red/blue coloring of edges of $G$.  Assume there is no blue induced star on $n$ edges.
Thus each vertex from $Y$ sends at least $\epsilon n$ red edges to $X'$. Since each blob has size $c \leq  n ( 1+ 2\epsilon) / |X|  $, we have that each vertex of $Y$ sends a red edge to at least $n_X= \frac{\epsilon  }{1+ 2\epsilon} |X|$ blobs. The total number of red edges is at least $\epsilon n |Y|$, thus there is a vertex, say  $x$,  in $X'$ that is adjacent to a set $Y_x$ of at least  $ \epsilon n |Y|/|X'| \geq \frac{\epsilon n |Y| }{(1+  2\epsilon)n} =  \frac{\epsilon}{1+2\epsilon}  |Y|$ vertices via red edges. Let $n_Y= \frac{\epsilon}{1+2\epsilon} |Y|$.\\

Consider $x$,  $Y_x$, a set of neighbours of $x$ adjacent to it via red edges, $y_1\in Y$, a vertex of $Y$ non-adjacent to $x$, $y_2\in Y_x$, and
sets $X_1$ and $X_2$, $X_1, X_2\subseteq X$, such that 
for any $v\in X_1$, $y_1$ sends a red edge to $B_v$ and for any $v\in X_2$, $y_2$ sends a red edge to $B_v$.
If there is a non-edge  $y' u$ between $Y_x$ and $X_1$ in $I$, we have a red induced $2K_2$  induced by $x, y', y_1, x'$, where $x'$ is a red neighbour of $y_1$ in $B_u$.
If there is a  non-edge  $y'' u$ between $Y_x$ and $X_2$ in $I$, we have a red induced $P_4$  induced by $ y'', x, y_2, x''$, where $x''$ is a red neighbour of $y_2$ in $B_u$.\\

We have that $|X_i| \geq n_X= \frac{\epsilon}{1+2\epsilon} |X|$, $i=1,2$, and $|Y_x|\geq  n_Y= \frac{\epsilon}{1+2\epsilon} |Y|$.
Thus, it is sufficient for us to check that in $I$ there is a non-edge between any subset of of $X$ of size $n_X$ and any subset of $Y$ of size $n_Y$.
Assume not, i.e., there is a subset $X''$ of $X$ of size $n_X$ and $Y''$ of $Y$ of size $n_Y$ so that $X''\cup Y''$ induces a complete bipartite graph in $I$.\\

We shall count the number $t$ of subsets of size $b-a$ containing $X''$, recalling that $|X|=b$ and $|X''| = \frac{\epsilon}{1+2\epsilon}b$:
$$t = \binom {b- |X''|}{b-a-|X''|} = \binom{b(1+  \epsilon)/(1+2\epsilon)}{a}.$$
Since any set in $Y''$ is adjacent to all of $X''$, we must have that $t\geq |Y''|$. 
We have that $$|Y''| \geq  \frac{\epsilon}{1+2\epsilon}|Y| = \frac{\epsilon}{1+2\epsilon}\binom{b}{b-a}.$$

Thus

\begin{eqnarray*}
t &\geq& |Y'| \Longrightarrow\\
\binom{b(1+  \epsilon)/(1+2\epsilon)}{ a}&\geq &\frac{\epsilon}{1+2\epsilon}\binom{b}{a} \Longrightarrow\\
\frac{b(1+  \epsilon)}{1+2\epsilon}\left(\frac{b(1+  \epsilon)}{1+2\epsilon} -1\right)\cdots \left(\frac{b(1+ \epsilon)}{1+2\epsilon} -a+1\right)&\geq& \frac{\epsilon}{1+2\epsilon} b \cdots (b-a+1)\Longrightarrow\\
\frac{b(1+ \epsilon)}{b}  \frac{b(1+ \epsilon) - 1(1+2\epsilon)}{b-1} \cdots \frac{b(1+  \epsilon) - (a-1)(1+2\epsilon)}{b- (a-1)}& \geq& (1+2\epsilon)^{a-1}\epsilon \Longrightarrow\\
\frac{b}{b}  \frac{b - 1(1+  \epsilon)^{-1}(1+2\epsilon)}{b-1} \cdots \frac{b - (a-1)(1+ \epsilon)^{-1}(1+2\epsilon)}{b- (a-1)} &\geq &
  \left(\frac{1+2\epsilon}{1+ \epsilon}\right)^{a-1}\epsilon.\\
\end{eqnarray*}

Note that the left hand side of the last inequality is less than $1$. On the other hand, since $\epsilon = 1/\sqrt{a}$, the right hand side is 
$ \left(\frac{1+2\epsilon}{1+ \epsilon}\right)^{a-1}\epsilon \geq    (1+ \frac{\epsilon}{1+\epsilon} (a-1)) \epsilon \geq  1$. 
This contradiction concludes the proof.
\end{proof}

\begin{proof} [Proof of Theorem \ref{Upper-bound}]
The theorem follows from Lemmas \ref{2K2-P4}  and \ref{Ktt}. 
\end{proof}

\vskip 1cm

\noindent
\section{Acknowledgements}  The authors thank Casey Tompkins for interesting discussions on $\IR(2K_2, K_{1,n})$ and  Christian Ortlieb for a nice observation on the lower bound in Lemma \ref{stars}. The first author thanks Ryan Martin for conversations about embeddings in multicolor multipartite graphs,  Jacob Fox and Stanford University for hospitality.

\end{document}